\documentclass{aims}
\usepackage{amsmath}
  \usepackage{paralist}
\usepackage{subcaption}
\usepackage{tikz}
\usepackage{mathtools}
\usepackage{comment}
\usepackage{bbm}
  \usepackage{graphics} 
  \usepackage{epsfig} 
\usepackage{graphicx}  \usepackage{epstopdf}
 \usepackage[colorlinks=true]{hyperref}
\hypersetup{urlcolor=blue, citecolor=red}

\def\currentvolume{X}
 \def\currentissue{X}
  \def\currentyear{200X}
   \def\currentmonth{XX}
    \def\ppages{X--XX}

\newtheorem{theorem}{Theorem}[section]

\newtheorem{lemma}[theorem]{Lemma}
\newtheorem{proposition}{Proposition}

\theoremstyle{definition}
\newtheorem{definition}[theorem]{Definition}

\newtheorem*{theorem*}{Theorem}

\title[The sandpile identity element] 
      {The sandpile identity element on an ellipse}

\author[Andrew Melchionna]{}


 \email{am2433@cornell.edu}

\thanks{The first author is supported by NSF grant DMS-1455272}

\thanks{$^*$ Corresponding author: Andrew Melchionna}

\begin{document}
\maketitle

\centerline{\scshape Andrew Melchionna$^*$}
\medskip
{\footnotesize
 \centerline{Cornell University Department of Mathematics}
   \centerline{301 Tower Rd.}
   \centerline{Ithaca, NY 14853, USA}
} 

\bigskip

 \centerline{(Communicated by the associate editor name)}

\begin{abstract}
We consider certain elliptical subsets of the square lattice. The recurrent representative of the identity element of the sandpile group on this graph consists predominantly of a biperiodic pattern, along with some noise. We show that as the lattice spacing tends to 0, the fraction of the area taken up by the pattern in the identity element tends to 1.
\end{abstract}

\section{Introduction}

First investigated in the 1980's by Bak, Tang and Wiesenfeld, the abelian sandpile model is governed only by simple local interaction rules, yet demonstrates interesting and well-synchronized behavior on the large scale. The model is as follows. Begin with a function $\sigma : \mathbb{Z}^2 \rightarrow \mathbb{Z}_{\geq 0}$, representing the number of grains of sand on the individual vertices in $\mathbb{Z}^2.$ We will refer to $\sigma$ as the \textit{initial configuration}.  If a site $x \in \mathbb{Z}^2$ has $\sigma(x) \geq 4,$ the site is deemed unstable, and must be 'toppled', in the following manner. Remove 4 grains of sand from the unstable site, and donate them, one each, to the site's 4 nearest neighbors on $\mathbb{Z}^2$. Continue adjusting the sandpile in this manner, performing these toppling moves at unstable sites until every site is stable, i.e. has fewer than 4 grains of sand. Toppling moves occur at successive discrete time steps.

Consider the following sandpile process. Given a finite subset $E \subset \mathbb{Z}^2,$ recall that the outer boundary $\partial E$ is equal to the set of vertices $x$ in $\mathbb{Z}^2 \backslash E$ such that $x$ is adjacent to $y$ for some $y$ in $E$. Initialize the sandpile with 0 grains of sand everywhere on $E$. We allow the outer boundary to be an infinite source of sand, in the following sense. For any fixed positive integer $n$, topple the outer boundary $n$ times, allowing the sites in $E$ which are adjacent to the boundary to accumulate sand from these toppling moves, and disregarding the sand which accumulates at any points not lying in $E$. Once the outer boundary has toppled $n$ times, a (possibly unstable) configuration has formed in $E$; in particular, the neighbors of the outer boundary which lie in $E$ have accumulated some sand, while the rest of the sites in $E$ still do not have any grains of sand on them (and we ignore any sand outside of $E$). Now proceed to stabilize the sandpile on $E$ in the usual way, ignoring any grains of sand which leave $E$. We emphasize that, after the $n$ topplings of the outer boundary, we do not ever topple sites in $E^c$ (the complement of $E$) again. One can think of $E$ as a tabletop; when a grain of sand falls off of $E$, it is lost forever. Note that the above discussion of stabilizing an unstable sandpile relies on the so-called 'abelian property' of the model, which states that any reasonable sequence of stabilizing toppling moves ultimately leads to the same unique stable configuration. This construction is made precise in  \hyperref[Preliminaries]{section 2}.

\begin{figure}
\label{patternfigure}
 \centering
  \subcaptionbox{}{\includegraphics[width=2in]{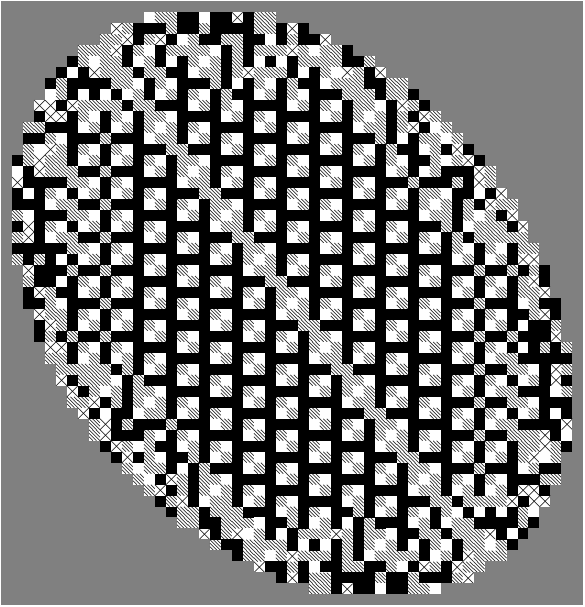}}\hspace{1.5em}
  \subcaptionbox{}{\includegraphics[width=2in]{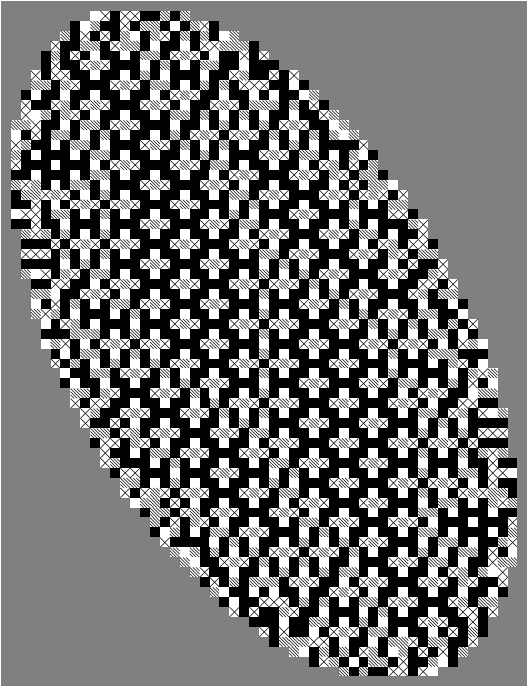}}\hspace{-0.8em} \newline
  \subcaptionbox{}{\includegraphics[width=2in]{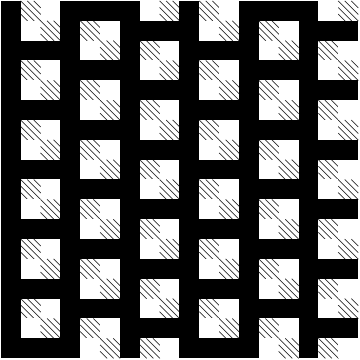}}\hspace{1.5em}
\subcaptionbox{}{\includegraphics[width=2in]{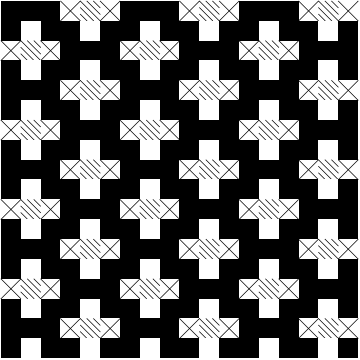}}\hspace{1em}

\caption{Figures 1a) and 1b) correspond to the identity elements for the ellipses $E_{A,k}$, with } $A = \begin{pmatrix}\frac{10}{9} & \frac{1}{3} \\[0.15cm] \frac{1}{3} & 1\end{pmatrix}$ and $\begin{pmatrix}\frac{4}{3} & \frac{1}{2} \\[0.15cm] \frac{1}{2} & \frac{3}{4}\end{pmatrix},$ respectively, with $k = 18^2$ in both. Figures 1c) and 1d) show the patterns $p_A$ corresponding to the ellipses in 1a) and 1b) respectively.  A black square represents a vertex with 3 grains of sand; dark patterned (parallel lines), 2; light patterned (cross), 1; white, 0.
\end{figure}

It is a fact \cite{corry2018divisors} that, given $E \subset \mathbb{Z}^2,$ there exists a positive integer $N$ such that the stabilized sandpile resulting from the above process above will be identical for all values of $n \geq N.$ This stable sandpile plays a crucial role, which we now illustrate. Consider the above process, in which we first topple the boundary $N$ times, and then stabilize the interior of $E$. Call the resulting sandpile configuration $e$. If we now use $e$ as an initial configuration and topple the boundary $k$ (with $k$ an arbitrary positive integer) more times, and then stabilize the interior of $E$ again, then every site in $E$ will topple exactly $k$ times, and the resulting stable sandpile will again be the configuration $e$. More generally, let $r: E \rightarrow \mathbb{Z}_{\geq 0}$ be any initial sandpile configuration. Topple the boundary $k \geq 0$ times and perform the subsequent stabilization. If every site in $E$ topples $k$ times in the stabilization and the resulting stable sandpile is again $r$ (and this holds for all $k \geq 0$), then we call the configuration $r$ \textit{recurrent}. The set of recurrent sandpiles on $E$ form a group under the operation of vertex-wise addition followed by stabilization. The stable sandpile $e$ is the identity element of this group, with the property that for any recurrent sandpile $r$, $\mathcal{S}(e + r) = r,$ where $+$ denotes vertex-wise addition, and $\mathcal{S}$ denotes the stabilization.  The identity element can be shown to satisfy a certain discrete boundary value problem, described in  \hyperref[BVP]{section 3}.

The goal of this paper is to characterize the identity elements of various elliptical subsets of $\mathbb{Z}^2$. That is, for a certain set of $2 \times 2$ symmetric matrices $\Gamma^+ \subset \text{Sym}(\mathbb{R}_{2\times 2})$ (defined in \hyperref[intsup]{section 2.3}) we consider the identity element on the graph 
\[
E_{A,k} = \{x \in \mathbb{Z}^2 : \frac{1}{2} x^T A x < k\},
\]
where $A \in \Gamma^+$ and $k >0$ is a real number. We show that the identity element predominantly features a biperiodic pattern $p_A$ associated to the matrix $A$ (this association is established in \hyperref[prop6]{Proposition 6}). \hyperref[patternfigure]{Figure 1} shows the identity elements for various $E_{A,k},$ along with the associated patterns $p_A.$ Note that the identity elements also  have some noise near the boundary of the ellipse and some one-dimensional defects in the interior of the ellipse. Our main result shows that for $A \in \Gamma^+$, the fraction of the area inside the ellipse $E_{A,k}$ which conforms to this periodic pattern tends to 1 in the limit of $k \rightarrow \infty$ (see  \hyperref[mainfigure]{Figure 2}). Our proof adapts the estimates developed by Pegden and Smart in \cite{pegden2020stability}. While \cite{pegden2020stability} characterizes a sandpile on a square subset of $\mathbb{Z}^2$ in which many different periodic patterns appear, we consider a sandpile on an elliptical geometry which isolates a single pattern. The main ingredient which is specific to the elliptical geometry is \hyperref[Andrewlemma]{Lemma 4.4}. Following \cite{pegden2020stability}, we define a point $x \in E_{A,k}$ to be $r$-good if the sandpile identity element of $E_{A,k}$ matches some translation of the pattern $p_A$ in the set $B_r(x) \cap \mathbb{Z}^2$ (see \hyperref[RGoodnessFigure]{Figure 7}). 

We now state a weak version of our \hyperref[theorem:mainresult]{Theorem 4.3}, which is the main result of this paper. 

\begin{theorem*}
Fix $A \in \Gamma^+.$ Fix $r$ to be a sufficiently large constant, dependent on $A$ (but not on $k$). Of the points which are distance at least $r$ from the boundary of the ellipse given by $A$ in $\mathbb{R}^2$, let $f(k,A)$ be the fraction of these points which are NOT r-good in the identity element of $E_{A,k}$. We then have that 
\[
\limsup_{k \rightarrow \infty} f(k,A) \cdot k^{1/4} \leq C_A,
\]
where $C_A$ is a constant depending on the matrix $A$. In particular, 
\[
\lim_{k \rightarrow \infty} f(k,A) = 0.
\]
\end{theorem*}
\begin{figure}
\label{mainfigure}
 \centering
  \subcaptionbox{$k = 8^2$}{\includegraphics[width=1.54in]{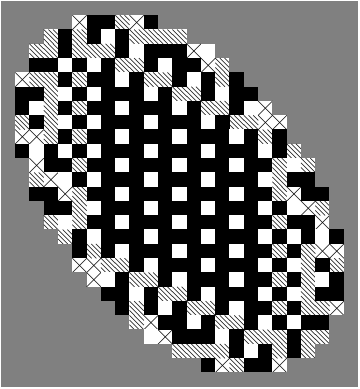}}\hspace{1em}%
  \subcaptionbox{$k = 16^2$}{\includegraphics[width=1.535in]{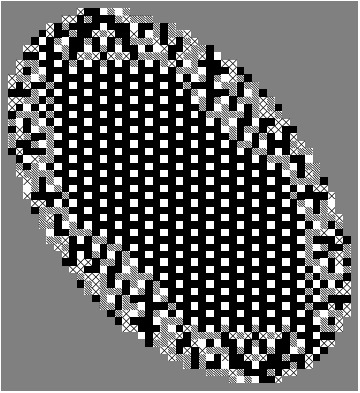}}\hspace{1em}%
  \subcaptionbox{$k = 24^2$}{\includegraphics[width=1.495in]{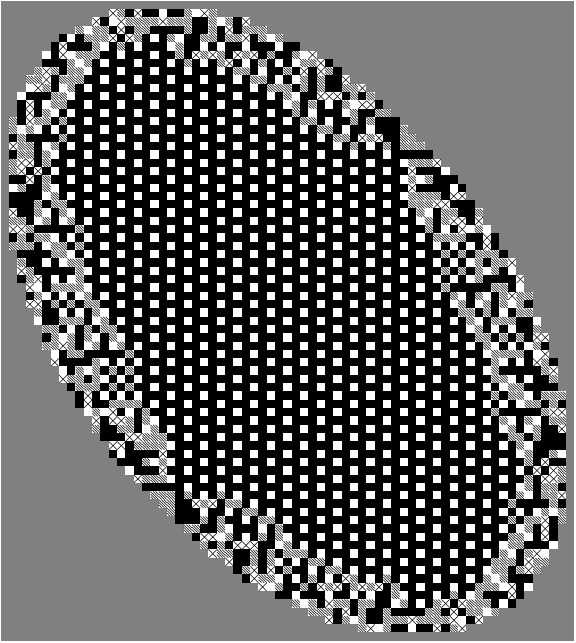}}\hspace{1em}%
\caption{The identity elements for the ellipses given by $E_{A,k}$, with} $A = \begin{pmatrix}\frac{5}{4} & \frac{1}{2} \\[0.15cm] \frac{1}{2} & 1\end{pmatrix}$ and various $k$. A black square represents a vertex with 3 grains of sand; dark patterned (parallel lines), 2; light patterned (cross), 1; white, 0. 
\end{figure}
The remainder of the paper is structured as follows: in section 2, we lay out the preliminaries of the abelian sandpile model and the Apollonian structure of the growth rates attainable by odometer functions. In section 3, we introduce the discrete boundary value problem which the stable sandpile solves. In section 4, we prove our main result, and discuss an extension of the theorem to the case of an uncentered ellipse. That is, we show that this theorem still holds (but now with different constant $C_A$) for the identity element of $E^p_{A,k},$ with this graph given by

\begin{figure}
\label{offcenterfigure}
 \centering
  \subcaptionbox{}{\includegraphics[width=2.04in]{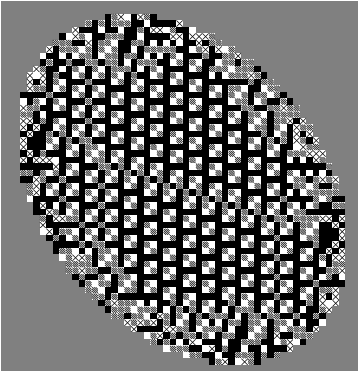}}\hspace{1.5em}%

\caption{The identity element for the graph $E^p_{A,k},$ with} $A = \begin{pmatrix}\frac{10}{9} & \frac{1}{3} \\[0.15cm] \frac{1}{3} & 1\end{pmatrix}$, $k=18^2,$ and $p = (.47,.5).$ 
\end{figure}
\[
E^p_{A,k} = \{x \in \mathbb{Z}^2 | \frac{1}{2} (x-p)^TA(x-p) < k \}
\]
for $p \in \mathbb{R}^2$. \hyperref[offcenterfigure]{Figure 3} shows the identity element for the graph $E^p_{A,k},$ with $p = (.47,.5),$ and with the same $A$ and $k$ values as found in Figure 1a.

\section{Preliminaries}
\label{Preliminaries}
\subsection{The sandpile on $\mathbb{Z}^2$}

Given an initial configuration $\sigma_0 :\mathbb{Z}^2 \rightarrow \mathbb{Z}_{\geq 0},$ let the (finite or infinite) sequence $(x_1,x_2,...)$ with $x_i \in \mathbb{Z}^2$ represent a sequence of toppling moves, with the vertex $x_i$ being toppled in the ith timestep. We demand that all vertices in $\mathbb{Z}^2$ are toppled finitely many times. Let the $\textit{odometer function}$ corresponding to the sequence $(x_1,x_2,...)$ be a function $u:\mathbb{Z}^2 \rightarrow \mathbb{Z}_{\geq 0}$ which counts the number of times each site in $\mathbb{Z}^2$ appears in the sequence. Note that, since $u(x)$ is finite for all $x \in \mathbb{Z}^2,$ we have that the toppling sequence gives a well-defined $\textit{final configuration}$, given by $\sigma_0(x) + \Delta u (x)$, where $\Delta$ is the discrete Laplacian operator, acting on functions with domain $\mathbb{Z}^2$. It is given by
\[
\Delta w (x) = \sum_{y \sim x} (w(y) - w(x) \big) = -4 w(x) + \sum_{y \sim x} w(y).
\]

The toppling sequence is called $\textit{legal}$ if a toppling move is only made when a vertex has 4 or more chips, and is called $\textit{stabilizing}$ if the resulting final configuration has fewer than 4 chips at every site on $\mathbb{Z}^2$. A configuration $\sigma$ is called $\textit{stabilizable}$ if it admits a stabilizing toppling sequence. A foundational result, justifying the the word 'abelian' in the name of the model, is that any two legal, stabilizing configurations give the same odometer function, and thus the same final sandpile:

\begin{proposition}[Abelian Property] \cite{fey2010growth}
Let $\sigma$ be an initial configuration, and suppose that there exists a stabilizing sequence $(x_1,...,x_n).$ Then there exists a legal stabilizing sequence, and any two legal stabilizing seqences are permutations of each other.
\end{proposition}

With the abelian property in mind, it is then natural to define an odometer function $u:\mathbb{Z}^2 \rightarrow \mathbb{Z}_{\geq 0}$ corresponding to a stabilizable configuration $\sigma,$ letting $u(x)$ represent the number of times a site x topples in a legal, stabilizing sequence. Given a stabilizable initial configuration $\sigma,$ we can then write the final, stable sandpile configuration $s : \mathbb{Z}^2 \rightarrow \mathbb{Z}_{\geq 0}$ in the following way:
\[
s(x) = \sigma(x) + \Delta u (x).
\]
By the abelian property, the odometer function $u$ corresponding to $\sigma$ is well-defined. The odometer function can be proven to satisfy the following 'least action principle':

\begin{proposition}[Least Action Principle] \cite{fey2010growth}
Let $u(x)$ be the odometer function for a stabilizable initial configuration $\sigma.$ Then $u(x)$ satisfies
\[
u(x) = \inf \{w(x) \, | \, w: \mathbb{Z}^2 \rightarrow \mathbb{Z}_{\geq 0}, \, \sigma (y) + \Delta w(y) \leq 3 \quad \forall  y \in \mathbb{Z}^2 \}.
\]
\end{proposition}
This proposition states that, during a stabilizing process, each vertex will topple as few times as necessary in order to stabilize the sandpile.

\subsection{The sandpile on a general graph}

Consider now a finite, connected, undirected multigraph $G = (V \cup \{q\},E)$, where $q$ is a \textit{sink vertex}. A sandpile configuration $s: V \rightarrow \mathbb{Z}$ is called \textit{stable} if $s(v) < \deg(v)$ for all $v$ in $V$. Define the graph Laplacian $\Delta_G$ similarly as above: for any integer-valued function on the vertices $w: V\cup\{q\} \rightarrow \mathbb{Z}$, we define the graph Laplacian $\Delta_G$ as
\[
\Delta_G w (x) = \sum_{y \sim x} \big( w(y) - w(x) \big),
\]
where $y \sim x \Leftrightarrow xy \in E.$ By enumerating the vertices in $V \cup \{q\},$ we can think of a function on the graph as a vector (with components corresponding to the value of the function on each vertex), and the graph Laplacian $\Delta_G$ as a matrix acting on these vectors. This matrix can be written as
\[
\big(\Delta_G\big)_{ij} =
\begin{cases}
-\deg{i} & i = j \\
M(i,j) & i \neq j
\end{cases}
\]
where $M(i,j)$ is the multiplicity of the edge connecting $i$ and $j,$ with $M(i,j) = 0$ if $ij \notin E.$

With this construction in mind, we define the reduced graph Laplacian $\tilde{\Delta}_G$ as the matrix obtained from $\Delta_G$ by deleting the row and column corresponding to the sink vertex $q.$

The set of sandpile configurations on $G$ enjoy a group structure, in the following sense. Consider the free abelian group $\mathbb{Z}^V,$ corresponding to the set of all possible sandpile configurations (allowing for negative amounts of chips) on the nonsink vertices, with vertex-wise addition. Consider also the equivalence relation on $\mathbb{Z}^V$ given by $w \sim u$ if and only if there exists a function $v: V \rightarrow \mathbb{Z}$ such that $w = u + \tilde{\Delta}_G v.$ $v(x)$ can be seen as the number of times that the vertex $x$ needs to be toppled in order to get from the configuration $u$ to the configuration $w.$ Note that we are allowing $v$ to take on negative values, corresponding to 'untopplings'. It can be shown that this defines an equivalence relation on $\mathbb{Z}^V.$ It can also be seen that this equivalence relation respects vertex-wise addition of two sandpiles. Thus, the quotient $\mathbb{Z}^V/\sim$ with the operation of vertex-wise addition is a group, called the $\textit{sandpile group}$ of the graph $G.$ 

We next define the notion of recurrency of a sandpile configuration.

\begin{definition}
Let $s: V \rightarrow \mathbb{Z}$ be a sandpile configuration.  Given a nonempty subset $X \subset V$, and an element $x \in X,$ define $\textbf{indeg}_X(x)$ to be the number of nearest neighbors of x which lie in X. A $\textbf{forbidden subconfiguration}$ is a nonempty subset $X \subset V$ such that, for all $x \in X,$ $s(x) < \text{indeg}_X(x)$. The sandpile $s$ is called $\textbf{recurrent}$ if it possesses no forbidden subconfigurations.
\end{definition}

The name 'recurrent' comes from the study of Markov chains on the space of sandpile configurations. In this framework, these configurations are recurrent in the sense that $\mathbb{P}^{s} (\# \{n \in \mathbb{N} : s_n = s\} = \infty ) = 1 $, where $(s_n)_{n \in \mathbb{N}}$ is a Markov chain on the space of stable sandpiles on $G$ which evolves by dropping a grain of sand uniformly at random on a nonsink vertex and stabilizing. $\mathbb{P}^{s}$ is the measure given by starting the Markov chain at $s_0 = s.$ We make use of the following standard fact linking recurrent configurations and the sandpile group:

\begin{proposition} \cite{jarai2014sandpile}
Every equivalence class in the sandpile group contains exactly one recurrent configuration.
\end{proposition}

The goal of this paper is to explore the recurrent representative of the identity element of the graph $E_{A,k}$, for certain matrices $A$ explored in the next section.

\subsection{Integer superharmonic matrices}
\label{intsup}
Let $\text{Sym}( \mathbb{R}_{2 \times 2})$ be the set of symmetric $2 \times 2$ matrices with real entries. Let \newline $q_A: \mathbb{Z}^2 \rightarrow \mathbb{Z}$ be defined by $q_A(x) = \frac{1}{2} x^T A x,$ where in the previous equation, $x \in \mathbb{Z}^2$ is considered as a 2-component vector. Define the set of integer superharmonic matrices, $\Gamma$, as follows:
\begin{align*}
\Gamma = \{A \in \text{Sym}_{2\times 2}(\mathbb{R})  \, | \, \exists o_A: \mathbb{Z}^2 & \rightarrow \mathbb{Z}, \, o_A (x) \geq q_A (x) + o (|x|^2), \, \\ &\Delta o_A (x) \leq 3 \, \, \forall x \in \mathbb{Z}^2\}.
\end{align*}

In words, a real, symmetric $2 \times 2$ matrix is integer superharmonic if there exists an integer-valued dominating (up to terms $o(|x|^2)$) function which satisfies $\Delta o_A (x) \leq 3$ for all $x$. Such a function will be called an 'integer superharmonic witness corresponding to A'. Note that the definition of a superharmonic function $o_A : \mathbb{Z}^2 \rightarrow \mathbb{Z}$ is usually that $\Delta o_A(x) \leq 0$ for all $x \in \mathbb{Z}^2,$ while we are using the convention that $o_A$ is required to satisfy $\Delta o_A(x) \leq 3$ for all $x \in \mathbb{Z}^2$. One can easily translate between these two conventions by simply adding a function such as $\frac{3}{2}x_1(x_1 + 1)$, which has Laplacian identically equal to 3. 

When considering a subset $X \subset \mathbb{Z}^2,$ we will speak of its outer and inner boundaries $\partial X$ and $\overline{\partial} X$, respectively, defined as

\begin{definition}
Given a subset $X \subset \mathbb{Z}^2,$ define
\[
\partial X = \{y \in \mathbb{Z}^2 - X \, | \, y \sim x \text{ for some } x \in X \}.
\]
\end{definition}

\begin{definition}
Given a subset $X \subset \mathbb{Z}^2,$ define
\[
\overline{\partial} X = \{y \in X \, | \, y \sim x \text{ for some } x \in \mathbb{Z}^2 - X \}.
\]
\end{definition}

\begin{definition}
Given a subset $X \subset \mathbb{Z}^2,$ define the complement of $X$ as
\[
X^c = \{y \in  \mathbb{Z}^2 \, | \, y \notin X \}.
\]
\end{definition}
We now define a subset $\Gamma^+ \subset \Gamma$ as follows:

\begin{definition} 
$\Gamma^+ = \{A \in \Gamma \, | \, \exists \epsilon > 0 \, \text{ s.t. }  A - \epsilon I \leq B \in \Gamma \implies B \leq A\}$.
\end{definition}

The matrices $A \in \Gamma^+$ are in some sense maximal, which we now discuss.

The set $\Gamma$ shares a relationship to an Apollonian circle packing in the plane, which we now discuss (following the authors of \cite{levine2017apollonian}). Consider the set of lines $\{x = 2k\}$ for $k \in \mathbb{Z}$, along with the set of circles $C_k$ of radius 1, centered at $(2k+1,0)$ for $k \in \mathbb{Z}.$ For any three pairwise tangent $\textit{general circles}$ (that is, circles or lines), there are exactly two $\textit{Soddy general circles}$ that are tangent to all three (we consider two lines to be tangent if and only if they are adjacent, i.e. are given by $\{x = 2k\}$ and $\{x=2k+2\}$ for some $k)$. The $\textit{Apollonian circle packing}$ generated by the lines $\cup_{k \in \mathbb{Z}} \{x = 2k\}$ and the circles $\cup_{k \in \mathbb{Z}} C_k$ is the minimal set of general circles which contains the generators and is closed under the addition of Soddy general circles for any pairwise-tangent triple in the packing (see  \hyperref[CoorpackFigure]{Figure 4}).

\begin{figure}
\label{CoorpackFigure} 
\input{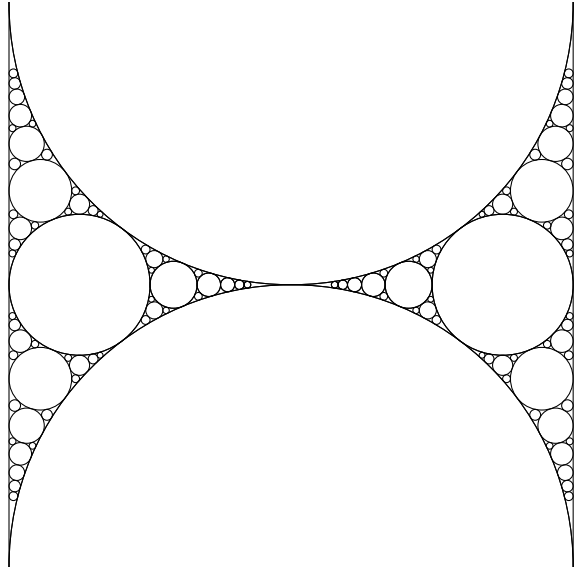}
\caption{ \cite{levine2017apollonian} A portion of the Apollonian circle packing between the lines $\{x=0\}$ and $\{x=2\}$}
\end{figure}

\begin{figure}
\label{ConeFigure}
\input{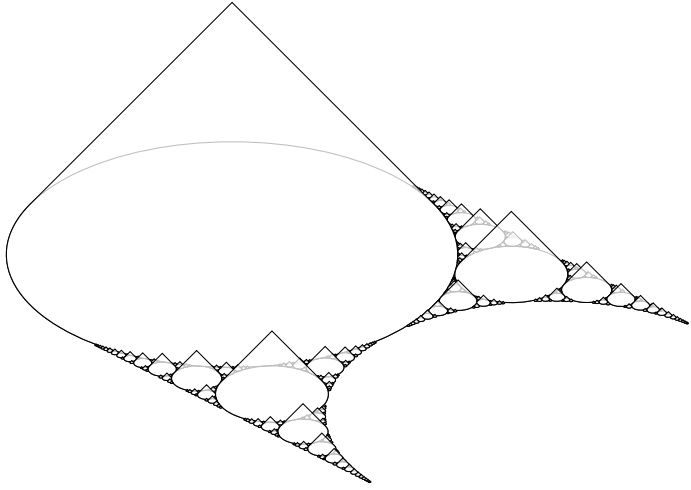}
\caption{\cite{levine2017apollonian} A portion of the boundary of the set $\Theta$}
\end{figure}

Let this Apollonian circle packing take place in the $\{z=2\}$ plane in $\mathbb{R}^3.$ Over each circle in the Apollonian circle packing, we can consider the cone protruding out of the plane (in the direction of positive-z) with slope 1, so that each circle is the base of a cone of height equal to the circle's radius. Define $\Theta \subset \mathbb{R}^3$ to be the downset of these cones, that is, the set of points $(x,y,z) \in \mathbb{R}^3$ such that $(x,y,z+\ell)$ is on a cone for some $\ell \in \mathbb{R}_{\geq 0}$ (see  \hyperref[ConeFigure]{Figure 5}). 

The authors of \cite{levine2017apollonian} prove that $\Gamma = \Theta$, under the following coordinate identification on $\Gamma$:
\[
\text{given } A = \begin{bmatrix} a & b \\ b & c \end{bmatrix}, \text{ set } \begin{cases} x = a-c \\ y = 2b \\ z = a + c \end{cases}.
\]
In other words,

\begin{proposition} \cite{levine2017apollonian}
$A \in \text{Sym}(\mathbb{R})_{2 \times 2}$ is integer superharmonic if and only if \newline $(x(A),y(A),z(A)) \in \Theta,$ with the coordinates $\big(x(A),y(A),z(A)\big)$ defined above. 
\end{proposition}

The matrices $A \in \Gamma^+$ correspond exactly to the peaks of the cones in $\Theta.$ A special property of matrices in $\Gamma^+$ is that they possess an integer superharmonic witness which is recurrent, in the following sense:

\begin{definition}
\label{recurrentfunc}
A function $v: \mathbb{Z}^2 \rightarrow \mathbb{Z}$ is recurrent in $X \subset \mathbb{Z}^2$ if $\Delta v \leq 3$ in X and 
\[
\sup_Y (v-w) \leq  \sup_{(X - Y) \cup \partial X}(v-w)
\]
whenever $w: \mathbb{Z}^2 \rightarrow \mathbb{Z}$ satisfies $\Delta w \leq 3$ in a finite $Y \subset X.$
\end{definition}

More precisely, the integer superharmonic witnesses for members of $\Gamma^+$ are recurrent on all of $\mathbb{Z}^2.$ We now give a proposition relating the definition of a recurrent sandpile with that of a recurrent function.

\begin{proposition}
\label{prop5}
A function $v: \mathbb{Z}^2 \rightarrow \mathbb{Z}$ is recurrent on a finite subset $X \subset \mathbb{Z}^2$ if and only if $\Delta v|_X$ is a recurrent sandpile on X. 
\end{proposition}

\begin{proof}
$\textbf{v recurrent function} \implies \Delta \textbf{v recurrent sandpile}$

Assume that $\Delta v$ has an FSC $Y \subset X,$ so that for all $x$ in $Y,$ $\Delta v (x) < \text{indeg}_Y(x)$. Equivalently, $\Delta v (x) + \text{outdeg}_Y(x)  < 4$, where $\text{outdeg}_Y(x) = 4-\text{indeg}_Y(x).$ Since $\text{outdeg}_Y (x) = \Delta \mathbbm{1}_{(X - Y) \cup \partial X}(x),$ we have that $w:= v+\mathbbm{1}_{(X - Y) \cup \partial X}$ satisfies $\Delta w \leq 3$ on $Y$ and 
\[
\sup_{Y} (v - w) = 0 > -1 = \sup_{(X - Y) \cup \partial X} (v-w),
\]
violating the recurrency of $v,$ and yielding a contradiction.

$\Delta \textbf{v recurrent sandpile} \implies \textbf{v recurrent function}$
Consider a function $w: X \cup \partial X \rightarrow \mathbb{Z}$ satisfying $\Delta w \leq 3$ in some finite $Y \subset X.$ Define $y_0$ to be a point at which $v-w$ is maximized in Y. Consider $\tilde{w}(x) = w(x) - w(y_0) + v(y_0)$, so that $(v-\tilde{w})(y_0) = 0,$ $\sup_Y (v - \tilde{w}) = 0,$ and $\Delta{\tilde{w}} = \Delta w.$ Define the set $Z \subset Y$ to be the maximal connected set containing $y_0$ on which $\tilde{w} - v =  0.$ Define $\tilde{w} \backslash v(x) := \max \big((\tilde{w}-v)(x),0\big)$. Clearly $\tilde{w} \backslash v$ is identically 0 on Z. 

Now, if $(\tilde{w} \backslash v) (x) \geq 1$ for all $x \in \partial Z,$ then we would have that 
\[
3 \geq \Delta \tilde{w} (z) \geq \Delta v (z) + \text{outdeg}_Z (z) \qquad \qquad \forall z \in Z,
\]
where the first inequality follows from our assumptions on $w,$ and the second inequality follows from the fact that $\Delta (w-v)(z) \geq \text{outdeg}(z),$ since $(\tilde{w} \backslash v) (x) \geq 1$ for all $x \in \partial Z.$ But this says exactly that $Z$ is a forbidden subconfiguration of X for $\Delta v$, a contradiction our assumption that $\Delta v$ is a recurrent sandpile. Thus, there must be some $x \in \partial Z$ such that $(\tilde{w} \backslash v) (x) = 0.$ By the definition of $Z,$ $x \notin Y,$ and thus $\sup_{(X - Y) \cup \partial X}(v - \tilde{w}) \geq 0 = \sup_Y (v - \tilde{w}),$ giving that $v$ is a recurrent function on $X.$

\end{proof}

In \cite{levine2017apollonian}, the authors present an explicit recurrent integer superharmonic witness for all $A \in \Gamma^+.$ A particularly interesting feature of such a witness is that the corresponding Laplacian is doubly periodic, and this period gives a hexagonal tiling of $\mathbb{Z}^2$, in the following way.

\begin{proposition} \cite{levine2017apollonian}
\label{prop6}
For every matrix $A \in \Gamma^+,$ there is a function $o_A : \mathbb{Z}^2 \rightarrow \mathbb{Z}$, a matrix $V \in \mathbb{Z}^{2 \times 3}$, and a subset $T \subset \mathbb{Z}^2$ such that the following hold:

1) $o_A$ is recurrent and can be decomposed as $o_A(x) = q_A(x) + L(x) + \rho_A(x) + c,$ where $q_A(x) = \frac{1}{2}x^T A x,$ $L(x) = b \cdot x$ for some $b \in \mathbb{R}^2,$ $c \in \mathbb{R},$ and $\rho_A: \mathbb{Z}^2 \rightarrow \mathbb{Z}$ is a $V \mathbb{Z}^{3}$ periodic function

2) If $x \sim y \in \mathbb{Z}^2$, then there is $z \in \mathbb{Z}^3$ such that $x,y \in T + V z.$

3) Let $z,w \in \mathbb{Z}^3.$ $(T + Vz) \cap (T + V w) \neq \emptyset $ if and only if $|z-w|_1 \leq 1.$ Furthermore,  $(T + Vz) \cap (T + V w) \subset  (\overline{\partial}T + Vz) \cap (\overline{\partial}T + V w)$ for any $z \neq w$.

4) $\cup_{z \in \mathbb{Z}^3} (\overline{\partial}T + Vz) \subset \{\Delta o_A = 3\}$

5) $1 \leq  |V|^2 \leq C \det (V)$, where $|V|$ is the $\ell_2$ operator norm of the matrix, $\det(V)$ is the determinant of the $2\times 2$ matrix formed by omitting the third column of $V$, and C is a universal constant.

6) $V 
\begin{bmatrix}
1 \\
1 \\
1\\
\end{bmatrix}
 = \begin{bmatrix}
0 \\
0 \\
\end{bmatrix}$
\end{proposition}

The above proposition ensures that the tile T, along with its translation $T + V z$ for $z \in \mathbb{Z}^3$, cover all of $\mathbb{Z}^2$, with overlap on the inner boundaries of adjacent tiles. The size of a tile is given by $(Tr(A)-2)^{-1}$ \cite{levine2017apollonian}. It also gives that $\Delta o_A = \text{Tr} (D^2 o_A)$ is periodic: for any $z \in \mathbb{Z}^3, \Delta o_A (x) = \Delta o_A (x + Vz).$ It also says that $\Delta o_A (x) = 3$ on the inner boundaries of all of the tiles. By fixing z in 3) and considering every w with $|z-w|_1 = 1$, we see that every tile in $\mathbb{Z}^2$ has exactly 6 neighboring tiles. That $\det (V) \neq 0$ implies that $\dim \Big(\text{Span}_{\mathbb{R}} \{ V_i\} \Big) = 2,$ where $\{V_i\}_{1 \leq i \leq 3}$ are the columns of V. Thus we can see that the vectors describing the periodicity of $\Delta o_A$ can be obtained by selecting any two columns of V. 

\begin{figure}
\label{BadEllipseFigure}
 \centering
  \subcaptionbox{$E_{A_1,k}$}{\includegraphics[width=2in]{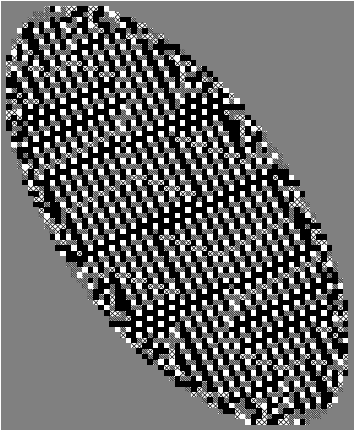}}\hspace{1em}%
  \subcaptionbox{$E_{A_2,k}$}{\includegraphics[width=2in]{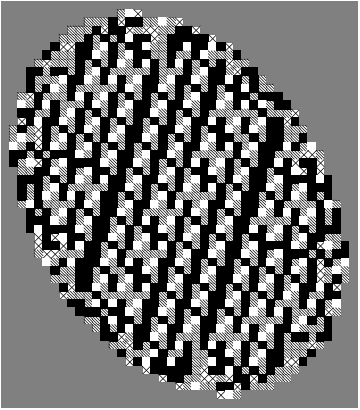}}\hspace{1em}

\caption{Identity elements of $E_{A_i,k}$, $A_i \notin \Gamma^+,$ $k = 18^2$ $A_1 \in \Gamma \backslash \partial \Gamma,$ $A_2 \notin \Gamma$.}
\end{figure}

The key property of matrices $A \in \Gamma^+$ which this paper explores is that the identity element of $E_{A,k}$ will predominantly feature the biperiodic pattern $p_A = \Delta o_A$. In this sense, the identity elements corresponding to these matrices feature a high level of order; matrices not belonging to $\Gamma^+$ will generally produce more chaotic identity elements with no evident patterns. \hyperref[BadEllipseFigure]{Figure 6} features the identity elements corresponding to various matrices which do not belong in $\Gamma^+$.  $k = 18^2$ for all identity elements in Figure 6, and the matrices are as follows:

\begin{align*}
A_1 &=  \begin{bmatrix} 1 & \frac{49}{100} \\  \frac{49}{100} & \frac{2}{3}\end{bmatrix} \in \Gamma \backslash \partial \Gamma \\
A_2 &=  \begin{bmatrix} \frac{7}{4} & \frac{45}{99} \\  \frac{45}{99} & \frac{4}{3}\end{bmatrix} \notin \Gamma,
\end{align*}
where $\partial \Gamma$ denotes the set of matrices represented by the Euclidean toplogical boundary of $\Theta \subset \mathbb{R}^3.$

\label{BVP}
\section{The Sandpile Boundary Value Problem}

Given a matrix $A \in \Gamma^+$ with $\det(A) > 0$ recall the following graph:
\[
E_{A,k} = \{ x \in \mathbb{Z}^2 : \frac{1}{2} xAx < k \}
\]
Next, form the graph $E_{A,k}'$ by considering the outer boundary $\partial E_{A,k}$ of the above graph, and identifying all vertices of $\partial E_{A,k}$ as one sink vertex. The sink vertex has edges connecting to $E_{A,k}$ according to the adjacencies of the vertices of the outer boundary $\partial E_{A,k}$ with the inner boundary $\overline{\partial} E_{A,k}$ (counting multiplicities).

The goal of the remainder of this paper is to characterize the recurrent representative of the sandpile identity element of $E_{A,k}'$ in the limit of large $k$. We first seek to find a suitable definition of an odometer function for the recurrent identity element on $E'_{A,k},$ a graph with sink.  We motivate our construction by considering the following discrete boundary value problem:

Find the pointwise-minimal function $v : \mathbb{Z}^2 \rightarrow \mathbb{Z}$ satisfying 
\begin{equation}
\begin{cases}
\Delta v (x) \leq 3 & x \in E_{A,k} \\
v(x) \geq 0 & x \in E_{A,k}^c.
\end{cases}
\end{equation}
Where pointwise-minimality is interpreted in the following way: 
Let 
\[
\label{odoms}
\mathcal{O} := \{w : \mathbb{Z}^2 \rightarrow \mathbb{Z} \quad | \quad w \text{ satisfies (1)}\}.
\]
Then $v(x) = \inf_{w \in \mathcal{O}} (w(x)).$ We make use of the following proposition:

\begin{proposition}
\label{prop7}
Let $u,w: \mathbb{Z}^2 \rightarrow \mathbb{Z}.$ If $\Delta u (x) \leq 3$ and $\Delta w (x) \leq 3$ for all $x \in \mathbb{Z}^2,$ then $v(x) := \min (u(x),w(x))$ also satisfies $\Delta v (x) \leq 3$ for all $x \in \mathbb{Z}^2.$
\end{proposition}

\begin{proof}
Let $x \in \mathbb{Z}^2$ be arbitrary. WLOG say that $v(x) = u(x).$ We then have
\[
\Delta v (x) = -4 v(x) + \sum_{y \sim x}v(y)  = -4 u(x) + \sum_{y \sim x}v(y) \leq -4 u(x) + \sum_{y \sim x}u(y), 
\]
since $v(y) \leq u(y)$ for all $y.$ The rightmost expression above is equal to $\Delta u (x),$ which is less than or equal to 3. 

\end{proof}

We now show that the solution, $v$, to this BVP is a recurrent function on $E_{A,k},$ giving that $\Delta v$ is a recurrent sandpile on $E'_{A,k}.$ 

\begin{proposition} The solution to the BVP (1) is a recurrent function on $E_{A,k}.$
\label{prop8}
\end{proposition}
\begin{proof}
Given $Y \subset E_{A,k}$ and a function $w$ satisfying $\Delta w \leq 3$ on $Y$, define the integer $w_0 := \sup_{\partial Y} (v - w).$ Then define the function $\tilde{w} = w + w_0,$ so that $\sup_{E_{A,k} \backslash Y \cup \partial E_{A,k}} (v-\tilde{w}) \geq 0.$ Note that nonnegativity of the $\sup$ follows from the definitions of $w_0$ and $\tilde{w}$, and the fact that $\partial Y \subset E_{A,k} \backslash Y \cup \partial E_{A,k}.$

Now, we would like to show that $\sup_{Y}(v - \tilde{w}) \leq 0.$ If it weren't, that is, if there were a $y_0 \in Y$ such that $\tilde{w}(y_0) < v(y_0),$ then the function 
\[
f(x) := \begin{cases}
\min(\tilde{w}(x),v(x)) & x \in Y \\
v(x) & x \in Y^c
\end{cases}
\]
would satisfy the BVP, contradicting that $v$ is the pointwise-least solution (since $f$ is strictly less than $v$). That $f$ satisfies the BVP can be seen by noting that, on $E_{A,k}^c \subset Y^c,$ $f(x) = v(x) \geq 0.$ Further, by \hyperref[prop7]{Proposition 7}, we have that $\Delta f \leq 3$ on $E_{A,k},$ since on $Y \cup \partial Y,$ $f(x) = \min(v(x), \tilde{w}(x))$ (note that $\tilde{w} \geq v$ on $\partial Y$). This gives that $\Delta f (x) \leq 3$ for all $x \in Y.$ Next, for $x \in Y^c,$ we have $\Delta f (x) \leq \Delta v (x),$ since $f = v$ at all $x \in Y^c$, and we have $f \leq v$ for points in $Y$ (which may neighbor points $x \in Y^c$). 
\end{proof}

Thus, by \hyperref[prop5]{Proposition 5}, $(\Delta v)|_{E_{A,k}}$ is a recurrent sandpile on $E_{A,k}.$ We now would like to show that the sandpile $(\Delta v)|_{E_{A,k}}$ belongs in the same equivalence class as the all-zeroes configuration (thus proving that it is the recurrent representative of the identity element). It suffices to find a function $w: E_{A,k} \rightarrow \mathbb{Z}$ such that $\Delta v + \tilde{\Delta}_G w = 0$ on $E_{A,k}.$ $-v$ is exactly this function.  Note that $v|_{\partial E_{A,k}}\equiv 0,$ since if there were a vertex $y \in \partial E_{A,k}$ such that $v(y) > 0,$ then we would have $(1-\mathbbm{1}_{x = y}) v \in \hyperref[odoms]{\mathcal{O}},$ since $v(x) \geq 0$ for all $x \in \partial E_{A,k}$ the condition $\Delta v(x) \leq 3$ for all $x \in E_{A,k}$ is preserved. This contradicts the fact that $v$ is the minimal function satisfying (1). Thus $v|_{\partial E_{A,k}} \equiv 0,$ which gives that $\tilde{\Delta}_G v = (\Delta v)|_{E_{A,k}}$. This gives
\[
\Delta v+ \tilde{\Delta}_G (-v) = \Delta v + \Delta (-v) = 0
\]
on $E_{A,k},$ by linearity of the Laplacian. Thus $\Delta v$ is the recurrent representative of the identity element. 

In an effort to make the difference between the operators $\Delta$ and $\tilde{\Delta}_G$ clear, we note that the former operator may in general include topplings of the boundary $\Delta E_{A,k},$ while the latter operator does not represent toppling of the outer boundary; only topplings of the interior. The two operators are equivalent when $v|_{\partial E_{A,k}} \equiv 0.$

In what follows, we refer to $v$ as the odometer function.

\section{Main results}

Fix a matrix $A \in \Gamma^+$ such that $\det (A) > 0$ (so that the level set of $\frac{1}{2} x^T A x$ is indeed an ellipse). Let $\lambda_1 < \lambda_2$ be the eigenvalues of the matrix A, with corresponding (unit-length) eigenvectors $v_1$, $v_2$. 

In the following, we only consider matrices lying in the fundamental domain of the Apollonian circle packing corresponding to $(x,y) \in [0,2]^2$ (see the discussion around Proposition 4). Matrices lying in this fundamental domain which also have positive determinant can be shown to satisfy $0 < \lambda_1 \leq 1 < \lambda_2$. A complete description of these matrices can be found in \cite{levine2016apollonian}. We use below that these matrices satisfy $\text{Tr} (A) >2,$ which follows immediately from the construction relating $\Gamma$ and $\Theta$ in the discussion before Proposition 4.

Since the matrix A is symmetric, $v_1$ and $v_2$ are orthogonal. Choose $v_1$ so that the angle $\theta$ that it makes with the x-axis is $\theta \in (-\pi,\pi],$ and choose $v_2$ so that $v_1 \times v_2$ points out of the page. Let $r_1$ and $r_2$ represent the semi-major and semi-minor axes of the ellipse $E_{A,k}$, respectively. They are given by $r_i = \sqrt{\frac{2k}{\lambda_i}}.$

Let $o_A$ represent a recurrent integer superharmonic witness for A, translated so that $o_A(0) = 0$. Let $v_{A,k}(x)$ be the solution to the boundary value problem in section 3, translated so that $v_{A,k}(x)|_{\partial E_{A,k}} \equiv k.$ In what follows, the subscripts A and k will sometimes be omitted to make the notation less cumbersome.

Experiments reveal that, except for some points near the boundary of $E_{A,k}$ and some 1-dimensional noise on the interior of the graph, the sandpile identity element of $E_{A,k}$ (i.e., the recurrent representative) almost perfectly matches the pattern given by $p_A = \Delta o_A = Tr A + \Delta rho_A$ (see \hyperref[prop6]{Proposition 6}). As k tends to infinity, the sandpile matches the pattern more and more closely; the noise takes up proportionally less area. The goal of this paper (in particular, of \hyperref[theorem:mainresult]{Theorem 4.3} below) is to quantify the convergence of the pattern given by $\Delta v_{A,k}$ to the pattern given by $\Delta o_A$ in the limit of $k \rightarrow \infty.$ 

We continue with notation before we state our main result.

\begin{definition}
\label{setdefinition}

Define the sets

\begin{align*}
&\tilde{E}_{A,k} :=  \{ x \in \mathbb{R}^2 : \frac{1}{2} x^TAx < k \} \subset \mathbb{R}^2 \\
&\tilde{E}_{A,k} \subset \tilde{F}_{L,A,k} := \{x \in \mathbb{R}^2: d(x, \tilde{E}_{A,k}) \leq L \} \subset \mathbb{R}^2, \\
& G_{L,A,k} := \{x \in E_{A,k}: d(x,\partial \tilde{E}_{A,k}) \geq L \} \subset E_{A,k} \subset \mathbb{Z}^2, \\
& \tilde{G}_{L,A,k} := \{x \in \tilde{E}_{A,k}: d(x,\partial \tilde{E}_{A,k}) \geq L \} \subset \tilde{E}_{A,k} \subset \mathbb{R}^2,
\end{align*}
\end{definition} 
where $L \in \mathbb{R}^+$, and $d(\cdot, \cdot)$ is the Euclidean distance. When referring to the set $\tilde{G}_{L,A,k}$ we often drop the subscripts $A$ and $k$ when what is meant is clear from the context.

Note that $E_{A,k} = \tilde{E}_{A,k} \cap \mathbb{Z}^2$ and $G_{L,A,k} = \tilde{G}_{L,A,k} \cap \mathbb{Z}^2.$ We will use the convention that if the name of a set has a tilde $\sim$ over it, then it is a subset of $\mathbb{R}^2,$ and if it doesn't, then it is a subset of $\mathbb{Z}^2.$

Consider an open ball $B_r(x)$ (in the Euclidean metric) around every point $x \in E_{A,k}.$ If on $B_r(x) \cap E_{A,k},$ the sandpile matches the $\Delta o_A$ pattern perfectly, we call this point r-good (see \hyperref[RGoodnessFigure]{Figure 7}).

\begin{definition}
A point $x_0 \in E_{A,k}$ is r-good if there exist $y, z \in \mathbb{Z}^2$ and $w \in \mathbb{Z}$ such that $v(x) = o_A(x+y) + z \cdot x + w$ for all $x \in B_r(x_0) \cap E_{A,k}$
\end{definition}

\begin{figure}
\label{RGoodnessFigure}
 \centering
  \subcaptionbox{The identity element of an ellipse}{\includegraphics[width=2in]{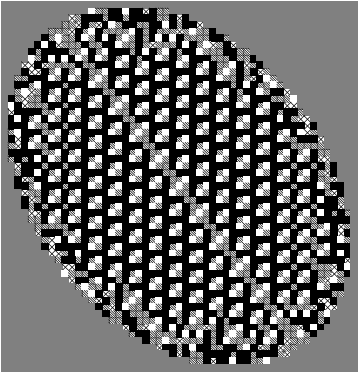}}\hspace{.5em}%
  \subcaptionbox{$x_1$ (white point) and $x_2$ (light gray point), with balls of radius $r$ centered at each}{\includegraphics[width=2in]{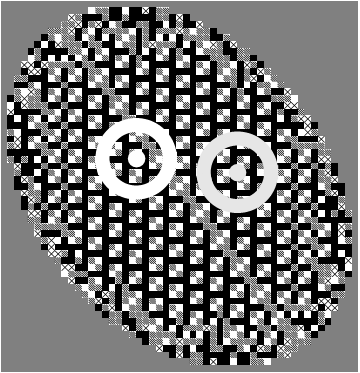}}\hspace{.5em}%
  \subcaptionbox{Enlarged portion of figure b)}{\includegraphics[width=3in]{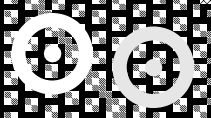}}\hspace{.5em}%
 
\caption{A demonstration of r-goodness. The white point $x_1$ is not r-good, while the light gray point $x_2$ is $r$-good.}
\end{figure}

In particular, by taking the Laplacian of both sides of the equation in the above definition, we see that if a point is r-good, then the sandpile pattern given by $\Delta v$ exactly matches some translation of the pattern $\Delta o_A$ in $B_r (x).$

\begin{theorem}
Take $r(k) = o(k^{1/4})$ and $r(k) \geq 3 |V|^3$ for all $k \in \mathbb{R}^+$. Let $f(k,A)$ be the fraction of points in $G_{r(k),A,k}$ which are $\textbf{not}$ r-good. Then

\[
\limsup_{k \rightarrow \infty} f(k,A) \cdot\frac{k^{1/4}}{r(k)} \leq C \cdot g(A),
\] 

where C is a universal constant, and $g(A)$ is a constant depending only on the matrix A, given by

\[
g(A) = 
\begin{cases} 

     \sqrt{ \lambda_1 + \lambda_2}( \sqrt{\lambda_1} + \sqrt{\lambda_2} ) \big( \sqrt{\frac{1}{\lambda_1 \lambda_2 (1+2\lambda_1 \lambda_2)}} + 2\sqrt{\frac{\lambda_1 \lambda_2}{1+2\lambda_1 \lambda_2}} \big) & \lambda_1 < \frac{1}{\sqrt{2}} \\
      
    ( \sqrt{\lambda_1} + \sqrt{\lambda_2}) \big(\frac{1}{\sqrt{\lambda_2}}+ \sqrt{2\lambda_2} \big) & \lambda_1\geq \frac{1}{\sqrt{2}} 

   \end{cases}.
\]
\end{theorem}
\label{theorem:mainresult}
In particular, the above implies that the fraction of $r(k)$-good points tends to 1, provided that $r(k)$ is $o(k^{1/4}).$ Note that we consider the fraction of points in $G_{r,A,k}$ (rather than $E_{A,k}$) in order to exclude points whose r-ball is not contained in the ellipse. The rate of $k^{1/4}$ appearing in the left hand side of the inequality in the theorem is a result of the $O(\sqrt{k})$ difference between the identity element's odometer function ($v$) and the integer superharmonic witness ($o$), detailed in \hyperref[Andrewlemma]{Lemma 4.4} below.

The proof is adapted from Theorem 10 in \cite{pegden2020stability}. We first use the recurrence of $o$ and $v$ to give an upper bound for $|v-o|$ on $E_{A,k}$ (Lemma 4.4). We then construct a 'touching map' whose range consists of good points (\hyperref[PSlemma]{Lemma 4.5}).  We finally estimate the area of the range of the touching map to give a lower bound on the number of r-good points.

\begin{lemma}
\label{Andrewlemma}
There exists a recurrent integer superharmonic witness $o$ for $A$ such that
\[
\sup_{x \in E_{A,k}}  |v(x) - o(x)|  \leq h^2(A) \sqrt{k} + o(\sqrt{k}),
\]
where
\[
h(A)^2 \coloneqq \begin{cases} 

      \sqrt{\frac{\lambda_1 + \lambda _2}{\lambda_1 \lambda_2 (1+2\lambda_1 \lambda_2)}} + 2\sqrt{\frac{\lambda_1 \lambda_2(\lambda_1 + \lambda _2)}{1+2\lambda_1 \lambda_2}}  & \lambda_1 < \frac{1}{\sqrt{2}} \\
      
      \frac{1}{\sqrt{\lambda_2}}+ \sqrt{2\lambda_2} & \lambda_1\geq \frac{1}{\sqrt{2}} 

   \end{cases}.
\]
\end{lemma}

\begin{proof}

Using \hyperref[prop6]{Proposition 6}, we write $o(x) = q(x) + L(x) + \rho(x)$, with $\rho(0) = 0$, and where we have dropped the subscript $A$ for convenience.

We first note that, since o and v are both recurrent on $E_{A,k}$ (see Propositions \hyperref[prop6]{6} and \hyperref[prop8]{8}), we have
\[
\sup_{E_{A,k}} |o-v| \leq \sup_{\partial E_{A,k}} |o-v|.
\]
Thus it suffices to show that $\sup_{\partial E_{A,k}} |o-v| \leq h^2(A) \sqrt{k} + o(k^{1/2}).$ We have that 
\[
 \sup_{x \in \partial E}|o-v|(x) \leq \sup_{x \in \partial E} (q_A(x) + |L(x)| - k) + \sup_{x \in \partial E} | \rho(x) | ,
\]
since $v|_{\partial E_{A,k}} = k$ and $q_A |_{\partial E_{A,k}} \geq k$. 

First note that $\sup_{x \in \partial E} | \rho(x) | = o(\sqrt{k}),$ since $\rho$ is a periodic function. We then seek to bound the term $\sup_{x \in \partial E} (q_A(x) + |L(x)| - k)$. Note the following inclusion of sets: $\partial E_{A,k} \subset \tilde{F}_{1,A,k}$. This is apparent from the fact that for any $x \in \partial E_{A,k},$ $d(x,\tilde{E}_{A,k}) \leq 1$, since $x \sim y$ for some $y \in E_{A,k}.$ 
0
For each $x \in \partial E_{A,k}$ we can assign a point $y_x \in \partial \tilde{F}$ to it by letting $y_x$ be the point on $\partial \tilde{F}$ which also lies on the line passing through the origin and $x$. See Figure 8.

\begin{figure}
\label{FFigure}
 \centering
\includegraphics[width=3in]{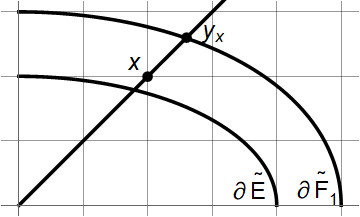}\hspace{1em}%

\caption{$x \in \partial E$ and the corresponding $y_x \in \partial \tilde{F}_1.$}
\end{figure}

Write $L(x) = b \cdot x = b_1 x_1 + b_2 x_2.$ Without loss of generality, we take $|b_1| \leq \frac{1}{2}$ and $|b_2| \leq \frac{1}{2}.$ If this $\textit{wasn't}$ the case, we could subtract some vector $b' \in \mathbb{Z}^2$, so that the modified superharmonic representative is still integer valued, and still has the appropriate growth at infinity. The Laplacian of the integer superharmonic witness is unaffected by the change in linear term. Thus, for all $x$, we have the inequality 
\[
|L(x)| \leq \|b\|\|x\| \leq \frac{1}{\sqrt{2}}\|x\| \leq \frac{1}{\sqrt{2}}\|y_x\|,
\]
where in the last inequality we have used that x lies between the origin and $y_x$ on the line which determines $y_x.$

We bound $Q(x)$ from above in a similar manner, noting that, as vectors in $\mathbb{R}^2,$ $y_x = (1 + \epsilon) x$ for some $\epsilon \geq 0.$ Thus 
\[
Q(x) = \frac{1}{2} x^T A x \leq \frac{ (1+\epsilon)^2}{2} x^T A x  = Q(y_x).
\]
Thus we have, for any $x \in \partial E_{A,k},$ that $Q(x) + |L(x)| \leq Q(y_x) + \frac{1}{\sqrt{2}}|y_x|.$ We now seek to maximize $Q(y) + \frac{1}{\sqrt{2}}|y|$ for $y \in \partial \tilde{F}.$

We now let $(x_1,x_2)$ parametrize $\mathbb{R}^2,$ and switch to coordinates $(x_1',x_2')$ in which the matrix is diagonalized, so that the semi-major axis aligns with the $x_1'$ axis, while the semi-minor axis aligns with the $x_2'$ axis.

We then have that
\[
E_{A,k} = \{x \in \mathbb{Z}^2 \, : \, \frac{1}{2} x'^T A' x' < k\},
\]
where 
\[
A' = \begin{bmatrix} \lambda_1 & 0 \\ 0 & \lambda_2 \end{bmatrix}.
\]

In what follows, we work in the primed coordinate system, but drop the primes for ease of notation.

It now suffices to bound $Q(z) + \frac{1}{\sqrt{2}}|z|$ for $z \in \partial \tilde{F}$ with $z_2 \geq 0$, that is, to consider only the top half of $\partial \tilde{F}$, by symmetries of $Q(z) + \frac{1}{\sqrt{2}}|z|$.

We first write the top half of the ellipse, $\partial \tilde{E},$ as a function of x. We have that
\[
x_2 = \sqrt{\frac{2k-\lambda_1x_1^2}{\lambda_2}}.
\]
We next note that, at any point on the top of $\partial \tilde{E}$, the outward facing unit normal vector is given by
\[
n(x_1,x_2) = \frac{1}{\Delta} \begin{bmatrix} \lambda_1x_1 \\ \lambda_2x_2 \end{bmatrix},
\]
where $\Delta := \sqrt{\lambda_1^2x_1^2 + \lambda_2^2 x_2^2}.$

Now, all of the points $z \in \partial \tilde{F}^+$ can be written as $z = (x_1,x_2) + n(x_1,x_2).$ Thus, we now parametrize $\partial \tilde{F}^+$ by its $x_1$-coordinate in the following way:
\[
\partial \tilde{F}^+ = \Bigg\{\Big(x_1\big(1 + \frac{\lambda_1}{\Delta}), \, \sqrt{\frac{2k-\lambda_1x_1^2}{\lambda_2}} \big(1 + \frac{\lambda_2}{\Delta}\big) \Big) \quad | \quad -\sqrt{\frac{2k}{\lambda_1}} \leq x_1 \leq \sqrt{\frac{2k}{\lambda_1}}\Bigg\}.
\]
Using this parametrization, we first consider the linear term, $\frac{1}{\sqrt{2}}|z|,$ as a function of $x_1$. We have

\begin{align*}
\frac{1}{\sqrt{2}}|z| &= \frac{1}{\sqrt{2}}\Big( x_1^2\big(1+ \frac{\lambda_1^2}{\Delta^2} + \frac{2\lambda_1}{\Delta} \big) + \frac{2k-\lambda_1x_1^2}{\lambda_2}\big(1 + \frac{\lambda_2^2}{\Delta^2} + \frac{2\lambda_2}{\Delta} \big) \Big)^{1/2} \\ &= \frac{1}{\sqrt{2}} \Big( 1 + \big(x_1^2 + \frac{2k - \lambda_1x_1^2}{\lambda_2}\big) + \frac{4k}{\Delta} \Big)^{1/2}.
\end{align*}
Next, we seek to write Q(z) in the same manner. We have
\begin{align*}
Q(z) &= \frac{1}{2} z^T A z = \frac{1}{2} \Big( \lambda_1x_1^2\big(1 +  \frac{\lambda_1^2}{\Delta^2} + \frac{2\lambda_1}{\Delta} \big) + (2k-\lambda_1x_1^2)\big(1 + \frac{\lambda_2^2}{\Delta^2} + \frac{2\lambda_2}{\Delta} \big) \Big) \\ &= k + \Delta + \frac{\lambda_1^3 x_1^2 + \lambda_2^3 x_2^2}{2 \Delta^2}
\end{align*}
Now we have, for $z = (x_1,x_2) \in \partial \tilde{F}^+,$ that
\begin{align}
 \frac{1}{\sqrt{2}} |z| &+ Q(z) - k   \\&=  \frac{1}{\sqrt{2}} \Big( 1 + \big(x_1^2 + \frac{2k - \lambda_1x_1^2}{\lambda_2}\big) + \frac{4k}{\Delta} \Big)^{1/2} +  \Delta + \frac{\lambda_1^3 x_1^2 + \lambda_2^3 x_2^2}{2  \Delta^2} \nonumber \\& \leq  \frac{1}{\sqrt{2}} \Big( 1 + \big(x_1^2 + \frac{2k - \lambda_1x_1^2}{\lambda_2}\big) + \frac{4\sqrt{k}}{C_2} \Big)^{1/2} +  \Delta + \frac{C_1 }{2 C_2^2 } \nonumber \\ & \leq \frac{1}{\sqrt{2}} \Big(\big(x_1^2 + \frac{2k - \lambda_1x_1^2}{\lambda_2}\big) \Big)^{1/2} +  \Delta +  \frac{C_1 }{2 C_2^2 } + \Big(1 + \frac{4 \sqrt{k}}{C_2} \Big)^{1/2} \nonumber,
\end{align}
where we have used the fact that there exist constants $C_1$ and $C_2$ (which depend on $\lambda_1$ and $\lambda_2$ but not on $k$) such that, for all $z \in \partial \tilde{F}^+,$ $\Delta \geq C_2 \sqrt{k}$ and $\lambda_1^3 x_1^2 + \lambda_2^3 x_2^2 \leq C_1 k$.

We need to find the maximum value of the $x$-dependent terms, that is, of
\[
M(x_1) := \sqrt{\frac{\frac{2k}{\lambda_2}+x_1^2(1-\frac{\lambda_1}{\lambda_2})}{2}} + \Delta
\]
for $|x_1| \leq \sqrt{2k/\lambda_1}$. We find that 
\[
M'(x_1) = \frac{1}{\sqrt{2\lambda_2}}\frac{(\lambda_2-\lambda_1)x_1}{\sqrt{2k+(\lambda_2-\lambda_1)x_1^2}} - \frac{\lambda_1(\lambda_2-\lambda_1)x_1}{\sqrt{2k\lambda_2 - \lambda_1(\lambda_2-\lambda_1)x_1^2}},
\]
which is nonsingular on $|x_1| \leq \sqrt{2k/\lambda_1}$. For $\lambda_1 < \frac{1}{\sqrt{2}}$, the solutions to $M'(x_1) = 0$ are the following:
\begin{align*}
x_1 &= 0 \\
\pm x_1 & = \pm \nu := \pm \sqrt{2k\lambda_2} \frac{\sqrt{1-2\lambda_1^2}}{\sqrt{-\lambda_1^2 + \lambda_1\lambda_2 - 2\lambda_1^3 \lambda_2 + 2\lambda_1^2 \lambda_2^2}}.
\end{align*}
When $\lambda_1 \geq \frac{1}{\sqrt{2}},$ the only solution is 
\[
x_1 = 0.
\]
We further compute
\begin{align*}
M''(x_1) = &-\frac{\lambda_1^2(\lambda_2-\lambda_1)^2x_1^2}{(2\lambda_2k-\lambda_1(\lambda_2-\lambda_1)x_1^2)^{3/2}} - \frac{\lambda_1(\lambda_2-\lambda_1)}{(2\lambda_2k-\lambda_1(\lambda_2-\lambda_1)x_1^2)^{1/2}} \\& - \frac{\lambda_2^{3/2}(x_1-\frac{\lambda_1x_1}{\lambda_2})^2}{\sqrt{2}(2k+(\lambda_2-\lambda_1)x_1^2)^{3/2}} + \frac{\lambda_2-\lambda_1}{\sqrt{2\lambda_2} \sqrt{2k+(\lambda_2-\lambda_1)x_1^2}}.
\end{align*}
Note that $M''$ is nonsingular on the domain.

\textbf{Case 1:} $\lambda_1 < \frac{1}{\sqrt{2}},$ \newline Note that $M''(\pm \nu) < 0$ (where we have used that $\lambda_2 > 1$). Thus at $\pm \nu,$ $M'$ must switch sign from positive to negative, giving that $\pm \nu$ are absolute maxima on the domain in the case of $\lambda_1 < \frac{1}{\sqrt{2}}.$

Plugging $\pm \nu$ into $M$ and simplifying, we get
\[
M(\pm \nu) =  \sqrt{\frac{\lambda_1 + \lambda _2}{\lambda_1 \lambda_2 (1+2\lambda_1 \lambda_2)}k} + 2\sqrt{\frac{\lambda_1 \lambda_2(\lambda_1 + \lambda _2)}{1+2\lambda_1 \lambda_2}k}.
\]
Dividing by $\sqrt{k},$ the result follows.

\textbf{Case 2:} $\lambda_1 \geq \frac{1}{\sqrt{2}}$ \newline

Using that $\lambda_1 \geq \frac{1}{\sqrt{2}}$ and $\lambda_2 > 1,$ it is easily seen that $M(0) > M (\pm \sqrt{\frac{2k}{\lambda_1}})$. Again plugging $x_1 = 0$ into $M$ and simplifying, we get that
\[
M(0) = \frac{\sqrt{k}}{\sqrt{\lambda_2}}+ \sqrt{2\lambda_2 k}.
\] 
Combining the above two cases with inequality (2) and the material preceding, we see that 
\vskip 4mm
\hfill $\displaystyle \sup_{x \in E_{A,k}} |v(x) - o(x)| \leq  h(A)^2 \sqrt{k} + o(\sqrt{k}).$ \qedhere
\end{proof}

The next lemma, due to Pegden and Smart, states that if the quadratically-lowered integer superharmonic witness touches $v$ from below at 0 in $B_R,$ then the point 0 is $C^{-1}R$-good for some universal constant $C$. The proof of this lemma, found in \cite{pegden2020stability}, makes use of the maximum principle used in the \hyperref[recurrentfunc]{definition} of recurrent functions.

\begin{lemma} \cite[Lemma 14]{pegden2020stability},
\label{PSlemma}
There is a universal constant $C > 1$ such that, if
\newline
1) $R \geq C |V|^3$ \newline
2) $v: \mathbb{Z}^2 \rightarrow \mathbb{Z}$  is recurrent in $B_R$ \newline
3) $\psi_y(x) = o(x) - \frac{1}{2} \frac{|V|^2}{R^2}|x-y|^2 + k$ for some $k \in \mathbb{R}$ \newline
4) $\psi_y$ touches $v$ from below at $0$ in $B_R$, that is, $\max_{B_R} \psi_y - v = 0$ and this maximum is achieved at $0$ \newline
then $0$ is $C^{-1} R$-good.
\end{lemma}

In order to estimate the fraction of r-good points, we will compare subsets of $\mathbb{Z}^2$ to their analogues in $\mathbb{R}^2,$ and use various geometric and measure-theoretic techniques in the continuum. We proceed by proving some properties of the sets $\tilde{G}$ and $\tilde{E}$ (which are subsets of $\mathbb{R}^2).$ 

\begin{lemma} $\tilde{G}_{L,A,k}$ is convex.
\label{convexlemma}
\end{lemma}

\begin{proof}
Let $P,Q \in \tilde{G}_{L,A,k}$ be arbitrary. We would like to show that for all points $S \in \overline{PQ},$ $B_L(S) \subset \text{cl} (\tilde{E}_{L,A,k})$ where $\overline{PQ}$ is the line segment connecting $P$ and $Q$, and cl denotes topological closure. Of course, $S \in \tilde{E}_{L,A,k},$ by convexity of $\tilde{E}_{L,A,k}$ and $\tilde{G}_{L,A,k} \subset \tilde{E}_{L,A,k}.$

Now, let $T$ be the midpoint of the line segment $\overline{PQ},$ and consider a coordinate parametrization $(x_1,x_2)$ of $\mathbb{R}^2$ where $T$ sits at the origin, and $P$ and $Q$ sit on the negative and positive $x_1$ axes, respectively. Consider also a topologically open rectangle $R \subset \mathbb{R}^2$ centered on the origin ($T$) with sides parallel to the coordinate axes, height $2 L$ and width $|\overline{PQ}|,$ so that the segment $\overline{PQ}$ bisects the rectangle horizontally. Note that each of the corners of $R$ are exactly a distance $L$ from one of the points $P$ or $Q$, so that the corners of $R$ are all contained in the closure of $\tilde{E}_{L,A,k}$. Then by convexity of $ \text{cl} (\tilde{E}_{L,A,k}),$ $R \subset  \text{cl} (\tilde{E}_{L,A,k}).$ This gives that for any point $x \in B_L(S) \cap R,$ $x \in  \text{cl} (\tilde{E}_{L,A,k}).$

Now consider a point $x \in B_L(S) \cap R^c,$ with coordinates $(x_1, x_2).$ Without loss of generality, let $x_1 \geq 0.$ Let $S$ and $Q$ have coordinates $(s_1,s_2)$ and $(q_1,q_2)$ respectively. It now suffices to show that $x$ is within distance $L$ of $Q,$ because $B_L(Q) \subset \text{cl} (\tilde{E}_{L,A,k}).$ We have:
\begin{align*}
d^2(x,Q) &= (x_1 - q_1)^2 + (x_2-q_2)^2 = (x_1 - q_1)^2 + (x_2-s_2)^2 \\ &\leq (x_1 - s_1)^2 + (x_2-s_2)^2 < L^2.  \qedhere
\end{align*}
\end{proof}

The following construction connects the relevant discrete sets with their continuum counterparts.

For any $x = (x_1,x_2) \in \mathbb{Z}^2,$ define $x^\Box \subset \mathbb{R}^2$ as the closed square of unit length centered on $x$, i.e.
\[
x^\Box = \big[x_1-\frac{1}{2},x_1+\frac{1}{2}\big] \times \big[x_2 -\frac{1}{2}, x_2+\frac{1}{2}\big] \subset \mathbb{R}^2.
\] 
Next, for any number $L,$ define the sets
\begin{align*}
A &= \{x \in \mathbb{Z}^2 \, : \, x^\Box \subset \tilde{G_L}\} \\
B &= \cap_{B_i \in \mathcal{B}} B_i,
\end{align*}
where 
\[
\mathcal{B} = \{C \subset \mathbb{Z}^2 \, : \, \cup_{x \in C} x^\Box \supset \tilde{G}_L \}.
\]
Further, define
\begin{align*}
\tilde{A} &= \cup_{x \in A} x^\Box \\
\tilde{B} &= \cup_{x \in B} x^\Box
\end{align*}
That is, $\tilde{A}$ is the largest union of unit squares (centered around lattice points) such that $\tilde{A} \subset \tilde{G_L},$ and $\tilde{B}$ is the smallest union of unit squares (centered around lattice points) so that $\tilde{B} \supset \tilde{G_L}.$ 

The following lemma is due to Levine:

\begin{lemma}
\label{construction}
 $|\tilde{B}| - |\tilde{A}| \leq 16|\partial \tilde{E}_k|$, where $|\partial \tilde{E}_k|$ is the length of the circumference of the ellipse and $|\tilde{A}|$ and $|\tilde{B}|$ are the areas of $\tilde{A}$ and $\tilde{B}$ respectively.
\end{lemma}

\begin{proof} First, note that for all $x \in B-A,$ $d(x,\partial \tilde{G}) \leq \frac{1}{\sqrt{2}},$ since $x^\Box \not\subset \tilde{G}.$ Thus, $B_1(x)$ intersects $\partial \tilde{G}$ in an arc of length at least $2-\sqrt{2},$ where $B_1(x)$ is the ball of radius 1 centered at $x$.  Define $\alpha_x := B_1(x) \cap \partial \tilde{G}$ (see \hyperref[ArcFigure]{Figure 9}). Next, note that for every point $z \in \partial \tilde{G},$ $z$ lies on at most 9 arcs $\alpha_x.$ This follows by noting that for any point $z \in \mathbb{R}^2,$ $B_1(z) \subset B_2 (z'),$ where $z'$ is a closest lattice point to z. $B_2(z')$ contains exactly 9 points, thus there are no more than 9 points with distance less than 1 from z. We then write, noting that $B-A$ is a finite set, and using $|B-A|$ to denote the number of points in $B-A$:
\begin{figure}
\label{ArcFigure}
 \centering
\includegraphics[width=3in]{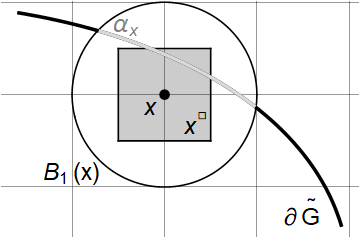}\hspace{1em}%
\caption{A schematic for Lemma 4.5. $x \in B - A$. }
\end{figure}
\[
(2-\sqrt{2}) |B-A| \leq \sum_{x \in B-A} |\alpha_x| = \int_{\partial \tilde{G}} \sum_{x \in B-A} \mathbbm{1}_{\alpha_x}(z) dz \leq 9 \int_{\partial \tilde{G}} dz = 9 |\partial \tilde{G}|.
\]

We conclude the proof by noting that the closures of the sets $\tilde{G}_k \subset \tilde{E}_k$ are compact and convex (\hyperref[convexlemma]{Lemma 4.6}). Thus $|\partial \tilde{G}| \leq |\partial \tilde{E}|$ (see, e.g., \cite{stefani2016monotonicity}). Since $ \frac{9}{2-\sqrt{2}} < 16,$ we are done. 
\end{proof}
The final lemma is a calculation of the area of the set $\tilde{G}.$
\begin{lemma}
\label{jacobianlemma}
For constant $L$ and for sufficiently large k, 
\[
|\tilde{E} - \tilde{G}_{Lr(k)k^{1/4}}| = Lr(k)k^{1/4} ( C_E - \pi L r(k) k^{1/4}),
\]
where $C_E$ is the circumference of the ellipse, and the vertical bars above refer to the Lebesgue measure of the continuum set.

\end{lemma}

\begin{proof}

Consider, for every point $x \in \partial \tilde{E}$, a line segment $\ell_x$ of length $L r k^{1/4}$ running inward normal to $\partial \tilde{E}.$ Let $z_x$ refer to the end of $\ell_x$ which lies inside the ellipse. We would like to prove that for sufficiently large $k,$ the end of each line segment will coincide with a point on the boundary of $\tilde{G}_{Lrk^{1/4}},$ and that these segments will be non-intersecting (i.e. $\ell_x \cap \ell_y = \emptyset$ for $x \neq y.$) Certainly, $\partial \tilde{G}_{Lrk^{1/4}}$ will not coincide with interior points of any line segment, lest the conditions defining $\tilde{G}_{Lrk^{1/4}}$ be violated. Then, $z_x$ will lie on $\partial \tilde{G}_{Lrk^{1/4}}$, provided that there does not exist a $y \in \partial \tilde{E}, y \neq x$ such that $d(y,z_x) < d(x,z_x).$ This will be satisfied if $Lrk^{1/4} < \inf_{p \in \partial \tilde{E}} r_p,$ where $r_p$ is the radius of curvature of $\partial \tilde{E}$ at a point $p \in \partial \tilde{E}$. The smallest radius of curvature occurs at the points on the semi-major axis, where $r_p = \frac{\sqrt{2k\lambda_1}}{\lambda_2}.$ Thus we require that $L r k^{1/4}<  \frac{\sqrt{2k\lambda_1}}{\lambda_2}.$ This criterion also suffices for non-intersection of the line segments. Of course, this condition will be met in the limit of large $k$, where we note that $r = o(k^{1/4}).$ We then use coordinates $(\theta,n)$, where $\theta$ gives points on the ellipse according to $x=\big(\sqrt{\frac{2k}{\lambda_1}}\cos(\theta),\sqrt{\frac{2k}{\lambda_2}}\sin(\theta)\big)$, and $n$ gives the position on the fiber $\ell_x,$ with $n=0$ coinciding with $\partial \tilde{E}$. Note that $\theta$ does not give the angle from the positive horizontal axis in general. In these coordinates, the Jacobian determinant becomes 
\[
J = \frac{-n r_1 r_2}{r_2^2 \cos^2(t) + r_1^2 \sin^2(t)} + r_1 r_2 \sqrt{\frac{\cos^2(t)}{r_1^2} + \frac{\sin^2(t)}{r_2^2}}.
\]
The area $|\tilde{E} - G_{Lrk^{1/4}}|$ then becomes the integral of the Jacobian over $0 \leq \theta < 2 \pi$ and $0 \leq n \leq Lr k^{1/4}$, giving $Lrk^{1/4} (C_E-L r k^{1/4} \pi),$ as desired. Note that the nonnegativity of the Jacobian (for sufficiently large $k$) on the region of integration follows from $r = o(k^{1/4}).$ 

\end{proof}

\begin{proof}[Proof of Theorem 4.3]
Consider the set $G_{7Ch(A)rk^{1/4},A,k} \subset \mathbb{Z}^2,$ with $h(A)$ defined in  \hyperref[Andrewlemma]{Lemma 4.4} (and referred to henceforth as $h$), C the constant from Lemma 4.5, and the set $G_{L,A,k}$ given in \hyperref[setdefinition]{Definition 4.1}. Consider the function
\[
\phi_y(x) = o(x) - \frac{1}{2}\frac{|V|^2}{(2Cr)^2}|x-y|^2
\]
for $y \in G_{7Chrk^{1/4}}$. Note that, for all sufficiently large $k$, $v(y) - \phi_y(y) = v(y) - o(y) \leq h^2  \sqrt{k} + o(\sqrt{k})$ (from Lemma 4.4). Furthermore, for $z \in E - B_{3h k^{1/4} \cdot 2Cr}(y),$ we have
\begin{align*}
(v-\phi_y)(z) &= (v-o)(z) + \frac{1}{2}\frac{|V|^2}{(2Cr)^2}|z-y|^2 \\&\geq -h^2 \sqrt{k}  + \frac{9}{2}|V|^2h^2 \sqrt{k} + o(\sqrt{k}) \\&\geq 3 h^2 \sqrt{k} + o(\sqrt{k}),
\end{align*}
where we've used that $|V|^2 \geq 1$ from \hyperref[prop6]{Proposition 6}.
Combining the two previous inequalities gives that, for sufficiently large $k$, $(v-\phi_y)$ attains its minimum in E on the set $ B_{6Ch k^{1/4} r}(y)$, say at the point $x_y.$ Thus $\phi_y$ can be translated to touch $v$ from below in $E$ at $x_y.$ In particular, for $k$ sufficiently large (so that $6Chk^{1/4} r(k) \geq 2 C r(k)$), $\phi_y$ can be translated to touch $v$ from below in $B_{2Cr(k)}(x_y)$ at $x_y$. We next apply  \hyperref[PSlemma]{Lemma 4.5} (taking $R = 2Cr(k)$) to see that $x_y$ is $2r(k)$-good. Following \cite{pegden2020stability}, we refer to the map $\theta: y \mapsto x_y$ as the 'touching map'.

Since, by Lemma 4.5, points in the range of the touching map are guaranteed to have a ball surrounding them which witness the matching of the sandpile identity element with the pattern given by $\Delta o_A$, we would like to estimate the number of points in the range of the touching map. If the touching map were injective, we would have that the number of good points is equal to $|G_{7Ch(A)rk^{1/4}}|$, the size of the domain of the touching map. However, this map is not injective in the standard sense; we instead make use of a weaker form of injectivity.

$\textbf{Claim \cite{pegden2020stability}:}$  For every $y \in G_{7Chrk^{1/4}},$ there are sets $y \in T_y \subset E$ and $x_y \in S_y \subset B_{|V|}(x_y)$ such that $|T_y| \leq |S_y|$ and $S_y \cap S_{\tilde{y}} \neq \emptyset$ implies $S_y = S_{\tilde{y}}$ and $T_y = T_{\tilde{y}}.$ 

Following \cite{pegden2020stability}, we choose a $\mathcal{Y} \subset G_{7Chrk^{1/4}}$ maximal subject to $\{S_y | y \in \mathcal{Y}\}$ being disjoint. We then have
\[
|\cup_{y \in \mathcal{Y}} S_y| \geq \sum_{y \in \mathcal{Y}} |T_y| \geq |\cup_{y \in \mathcal{Y}} T_y| \geq |G_{7Chrk^{1/4}}|,
\]
where we've used that $|S_y| \geq |T_y|$ for all y, the subadditivity of the measure, and the above claim for each of the inequalities, respectively. Lastly, note that for all $y$, every point in $S_y$ is r-good. This follows because every $S_y$ contains a 2r-good point, the diameter of $S_y$ is at most $2|V|,$ and  $r \geq 3|V|.$ We thus have that the fraction of good points in $G_r$ is at least $\frac{|G_{7Chrk^{1/4}}|}{|G_r|}$. 

We now aim to estimate this fraction by instead considering the areas of the continuum counterparts of the above sets; $|\tilde{G}_{7Chrk^{1/4}}|$ and $|\tilde{G}_r|$. Note that
\[
|G_L| - |\tilde{G}_L| \leq |B| - |\tilde{A}| =|\tilde{B}| - |\tilde{A}| \leq 16 |\partial \tilde{E}_k|
\]
and
\[
|\tilde{G}_L|-|G_L| \leq |\tilde{B}| - |A| = |\tilde{B}| - |\tilde{A}| \leq 16 |\partial \tilde{E}_k|,
\]
where $A, B, \tilde{A}, \tilde{B}$ are given in the construction before \hyperref[construction]{Lemma 4.7} for fixed $L$, and we appeal directly to Lemma 4.7 for the last inequality in each of the above two lines. Thus $\bigl||\tilde{G}_L| -|G_L| \bigr| \leq 16 |\partial \tilde{E}_k|$. This gives:
\begin{align*}
\frac{|G_{7Chrk^{1/4}}|}{|G_r|}& \geq \frac{1}{|G_r|} (|\tilde{G}_{7Chrk^{1/4}}| -  \bigl| |G_{7Chrk^{1/4}}| - |\tilde{G}_{7Chrk^{1/4}}| \bigr|)\\& \geq  \frac{1}{|G_r|} (|\tilde{G}_{7Chrk^{1/4}}| -  16 |\partial \tilde{E}_k|) \geq \frac{|\tilde{G}_{7Chrk^{1/4}}| -  16 |\partial \tilde{E}_k|}{|\tilde{G}_r| + \bigl| |\tilde{G}_r |-|G_r|\bigr|}
\end{align*}

We now consider the fraction of points which are \textbf{not} r-good. By the above, this fraction is bounded from above by 
\begin{align*}
&\frac{|\tilde{G}_r| + \bigl| |\tilde{G}_r |-|G_r|\bigr| - |\tilde{G}_{7Chrk^{1/4}}| + 16 |\partial \tilde{E}_k|}{|\tilde{G}_r| + \bigl| |\tilde{G}_r |-|G_r|\bigr|} \leq \frac{|\tilde{G}_r| - |\tilde{G}_{7Chrk^{1/4}}| + 32 |\partial \tilde{E}_k|}{|\tilde{G}_r|} 
\\&= \frac{ |\tilde{E}_k - \tilde{G}_{7Chrk^{1/4}}|- |\tilde{E}_k - \tilde{G}_r|}{|\tilde{E}_k|- |\tilde{E}_k - \tilde{G}_r|} + \frac{32|\partial \tilde{E}_k|}{|\tilde{E}_k|- |\tilde{E}_k - \tilde{G}_r|}
\end{align*}

Appealing now to \hyperref[jacobianlemma]{Lemma 4.8}, the fraction of r-bad points in $G_r$ is at most
\[
\frac{7Chrk^{1/4} C_E - \pi (7Chr)^2 k^{1/2}- r C_E + \pi r^2}{\pi \frac{2k}{\sqrt{\lambda_1 \lambda_2}} -r C_E + \pi r^2 } +  \frac{32|\partial \tilde{E}_k|}{\pi \frac{2k}{\sqrt{\lambda_1 \lambda_2}} -r C_E + \pi r^2}
\]
Now, using the estimate $C_1(r_1+r_2) \leq C_E \leq C_2(r_1+r_2)$ (for some constants $C_1$ and $C_2$) and manipulating, we get: 
\begin{align*}
f(k,r) &\leq \frac{(7ChrC_2 - \frac{C_1 r}{k^{1/4}})\sqrt{\frac{2}{\lambda_1 \lambda_2}}(\sqrt{\lambda_1}+\sqrt{\lambda_2}) + \pi (\frac{r^2}{k^{3/4}} - \frac{49C^2h^2r^2}{k^{1/4}})}{\frac{2\pi}{\sqrt{\lambda_1 \lambda_2}}k^{1/4} - \frac{r}{k^{1/4}} C_2 \sqrt{\frac{2}{\lambda_1 \lambda_2}} (\sqrt{\lambda_1}+\sqrt{\lambda_2}) + \frac{r^2}{k^{3/4}} \pi} \\&+ \frac{1}{k^{1/4}} \frac{32 C_2 \sqrt{\frac{2}{\lambda_1 \lambda_2}} (\sqrt{\lambda_1} + \sqrt{\lambda_2})}{\frac{2\pi}{\sqrt{\lambda_1 \lambda_2}}k^{1/4} - \frac{r}{k^{1/4}} C_2 \sqrt{\frac{2}{\lambda_1 \lambda_2}} (\sqrt{\lambda_1}+\sqrt{\lambda_2}) + \frac{r^2}{k^{3/4}} \pi}
\end{align*}
Multiplying by $\frac{k^{1/4}}{r}$ and taking the limit as $k \rightarrow \infty,$ we obtain
\[
\limsup_{k \rightarrow \infty} f(k,r) \cdot \frac{k^{1/4}}{r} \leq C  h (\sqrt{\lambda_1} + \sqrt{\lambda_2}),
\]
where we have absorbed all constants into $C.$ This proves the result. \qedhere
\end{proof}

We now wish to generalize our result slightly. Note that the graph $E_{A,k}$ is formed by taking the intersection of the square grid with an ellipse centered at the origin. We note that our result can be extended to ellipses centered at an arbitrary point in $\mathbb{R}^2,$ with Theorem 4.3 still holding (only with different, weaker constants $g(A),$ see \hyperref[offcenterfigure]{Figure 3}).

When replacing $(0,0)$ with an arbitrary centerpoint $p = (p_1,p_2) \in \mathbb{R}^2,$ we note that only Lemma 4.4 changes, with the rest of the proof of Theorem 4.3 being carried out as above. We now discuss the adaptations that can be made to Lemma 4.4 in order to accommodate the recentered ellipse.

First note that the statement
\begin{equation}
\sup_{E^p} |o-v| \leq \sup_{\partial E^p} |o-v| \leq -k +  \sup_{\partial E^p} \big(  q(x) + |L(x)| \big)  + \sup_{\partial E^p} p(x)
\end{equation}
still holds, as both $v$ and $o$ are still recurrent on $E^p,$ and $v \equiv k$ on $\partial E^p$ by definition. Without loss of generality, we take $p \in [-\frac{1}{2},\frac{1}{2}) \times [-\frac{1}{2},\frac{1}{2}),$ since we can always translate the witness $o_A,$ and $v$ is the solution to a $\mathbb{Z}^2-$translation-invariant BVP. 

Consider the second term in the rightmost expression above. As in \hyperref[Andrewlemma]{Lemma 4.4}, we can bound the linear function $L(x)$ by $\frac{1}{\sqrt{2}} |x|.$ Next consider the quadratic term $q(x).$ Define $q'(x) = \frac{1}{2} (x-p)^T A (x-p).$ We then have
\[
q(x) = q'(x) - \frac{1}{2} p^TA p + p^T A x \leq q'(x) + p^T A x \leq q'(x) + \frac{1}{\sqrt{2}} \lambda_2 |x|,
\]
where we have used Cauchy-Schwarz and that $\lambda_2$ is the largest eigenvalue of $A$ in the last inequality. We now have
\[
 \sup_{\partial E^p} \big(  q(x) + |L(x)| \big) \leq  \sup_{\partial E^p} \big(  q'(x) + \frac{1}{\sqrt{2}} (1 + \lambda_2) |x| \big) \leq \sup_{\partial E^p}   q'(x) +  \sup_{\partial E^p} \frac{1}{\sqrt{2}} (1 + \lambda_2) |x|.
\]

Heuristically speaking, the worst-case scenario for $q'(x)$ occurs when there is a point $y_x$ lying on $\partial E^p$ with unit distance from $x \in E^p,$ with both points lying along the semi-minor axis. This gives 
\[
q'(y_x) = \frac{1}{2}(r_2+1)^2 v_2^TA v_2 = \frac{\lambda_2}{2}(\sqrt{\frac{2k}{\lambda_2}}+1)^2 = k + \sqrt{2 k \lambda_2} + \frac{\lambda_2}{2}
\]
Plugging back into equation (3), we see that
\begin{align*}
\sup_{E^p} |o-v| &\leq \sqrt{2 k \lambda_2} + \sup_{\partial E^p} \frac{1}{\sqrt{2}} (1 + \lambda_2) |x| + \sup_{\partial E^p} p(x) + \frac{\lambda_2}{2}
\end{align*}
On $\partial E^p,$ $|x|$ can be bounded by $\sqrt{\frac{2k}{\lambda_1}} + 1 + \frac{1}{\sqrt{2}}$. This is due to the fact that, for a boundary point $x \in \partial E^p,$ $d(0,x) \leq d(0,p) + d(p,x).$ The terms $\sqrt{\frac{2k}{\lambda_1}} + 1$ come from $d(p,x),$ as it is the semimajor axis of the ellipse plus one, which accommodates for the farthest that a boundary point can be away from the interior of the ellipse. The $\frac{1}{\sqrt{2}}$ term bounds $d(p,0).$  

This gives that 
\[
\sup_{E^p} |o-v| \leq \big(\sqrt{2\lambda_2} + \frac{1 + \lambda_2}{\sqrt{\lambda_1}}\big) \sqrt{k}+ o(\sqrt{k}).
\]
The proof of Theorem 4.3 now proceeds as above, with the above inequality taking the place of Lemma 4.4, and the constant $\sqrt{2\lambda_2} + \frac{1 + \lambda_2}{\sqrt{\lambda_1}}$ taking the place of $h^2$.

\section{Open Questions}

Related open questions involve identity elements for non-maximal integer superharmonic ellipses. For example, take two maximal matrices $A$ and $B$, corresponding to Apollonian circles which are tangent. Consider a straight line running between these two centers in the $\{z=2\}$ plane, and then consider the inverse image of the projection operator from $\partial \Gamma$ onto the $\{z=2\}$ plane. What do the identity elements of matrices lying on this curve in $\partial \Gamma$ look like? In particular, what does the identity element of the matrix at the point of tangency look like? 

As we deviate from the peak of a cone (a maximal matrix) down the side of the cone in $\partial \Gamma$, experiments reveal that defects in the identity element appear in the form of stripes (\hyperref[BadEllipseFigure]{Figure 6}). Future work may include explorations of the rates at which stripes appear, and a characterization of the stripes appearing as we descend from a peak in $\partial \Gamma$ in arbitrary directions down the cone. 
\begin{figure}
\label{circlefigure}
 \centering
  \includegraphics[width=1.5in]{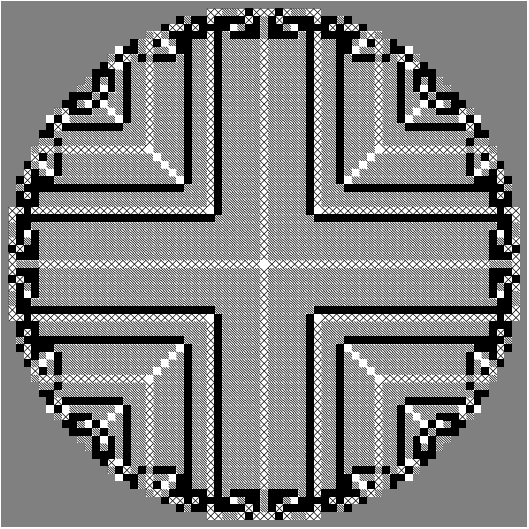}\hspace{.5em}%
 
\caption{The identity element of $B_{1152}(0) \cap \mathbb{Z}^2$.}
\end{figure}
We may also attempt to characterize the identity element of a circular region of the square grid. The matrix $\mathbb{I}_2$ which gives the circle is not itself a maximal integer superharmonic matrix, but is the limit of such matrices. Experiments reveal that the identity element is a constant background of $\Delta v = 2$, with some line-shaped defects in the interior of the circle (see \hyperref[circlefigure]{Figure 10}). Note further that vertical and horizontal lines only feature sites with 1 or 3 grains of sand, while the diagonal lines have 0 grains of sand.

\section{Acknowledgments}
The author thanks Lionel Levine, who conjectured this result, for his guidance throughout the process of writing this paper. The author also wishes to thank Swee Hong Chan for his helpful discussions, as well as the authors of \cite{levine2017apollonian} for their permission to use Figures 4 and 5. The author was partially supported by NSF grant DMS-1455272.

\medskip
Received xxxx 20xx; revised xxxx 20xx.
\medskip

\end{document}